\documentclass[12pt]{amsart}

\usepackage{amsmath,amssymb}
\usepackage{amsfonts}
\usepackage{amsthm}
\usepackage{latexsym}
\usepackage{graphicx}
\usepackage{epsfig}

\newtheorem{theorem}{Theorem}[section]
\newtheorem{lemma}[theorem]{Lemma}
\newtheorem{proposition}[theorem]{Proposition}
\newtheorem{definition}[theorem]{Definition}

\newtheorem{prop}{Proposition}[section]
\newtheorem{remark}[prop]{Remark}

\makeatletter \@addtoreset{equation}{section} \makeatother

\def\RR{{\mathrm R}}
\def\WW{{\mathrm W}}

\def\EE{{\mathrm E}}

\def\Rc{{\mathrm {Rc}}}

\def\SS{{\mathrm S}}

\begin{document}

\title{On Closed Manifolds with Harmonic Weyl Curvature}



\author{Hung Tran}

\address{Department of Mathematics and Statistics,
Texas Tech University, Lubbock, TX 79409}

\renewcommand{\subjclassname}{%
  \textup{2000} Mathematics Subject Classification}

\date{\today}

\begin{abstract} We derive point-wise and integral rigidity/gap results for a closed manifold with harmonic Weyl curvature in any dimension. In particular, there is a generalization of Tachibana's theorem for non-negative curvature operator. The key ingredients are new Bochner-Weitzenb\"{o}ck-Lichnerowicz type formulas for the Weyl tensor, which are generalizations of identities in dimension four. 
\end{abstract}
\maketitle
\tableofcontents

\section{\textbf{Introduction}}
Let $(M^n,g)$ be a closed manifold. $\RR, \WW, \Rc, \EE, \SS$ denote the Riemann curvature, Weyl tensor, Ricci tensor, traceable Ricci, and scalar curvature, respectively. In fact, the Riemann curvature decomposes into three orthogonal components, namely the Weyl curvature, a multiple of $\EE\circ g$, and a multiple of $g\circ g$ ($\circ$ denotes the Kulkani-Nomizu product). All could be seen as operators on the space of two-forms. 

The study of the Einstein equation ($\Rc$ is constant) and its various curvature generalizations have a vast literature, see \cite{besse} for an overview. The Einstein condition says that the trace-less Ricci part vanishes and scalar curvature is constant; thus, in dimension at least four, the problem of understanding the Riemann curvature on Einstein manifolds reduces to investigate the Weyl tensor. 

As a consequence, this paper aims to study a structure which generalizes the Einstein condition and involves the Weyl tensor. The outcomes are rigidity results.  

 
For an Einstein manifold, S. Tachibana showed that nonnegative curvature operator implies locally symmetric \cite{tachibana}. Recently, S. Brendle improved the result by only assuming nonnegative isotropic curvature \cite{brendle10einstein}. Those curvature assumptions are pointwise pinching $\WW$ against scalar curvature. Similarly, integral pinching of $\WW$, in terms of its $L^{n/2}$-norm, was also investigated, such as \cite{singer92, is02, catino15integral}  

A generalization of the Einstein condition is harmonic curvature: the former is $\Rc=\lambda g$ while the latter only requires $\Rc$ to be a Codazzi tensor. This condition is also of interests due to its connection to the theory of Yang-Mills equation. As a result, the topic has been studied in a wide array of literature, such as \cite{derd82, kim11, hv96, tv05, fu15pinched}. In particular, rigidity results due to integral pinched conditions were obtained in \cite{hv96, fu15pinched}. Also, A. Gray \cite{gray78} deduced a classification under a point-wise assumption. 

In connection with our interests, harmonic curvature is characterized by harmonic Weyl curvature (divergence of $\WW$ is vanishing) and constant scalar curvature. Thus, it is natural to investigate the condition of harmonic Weyl curvature separately. In this area, most of the research has focused on the case of K\"{a}hler or dimension four. A K\"{a}hler manifold with that condition must have parallel Ricci tensor \cite{gray78, tanno72, mat72}. In dimension four, Weyl decomposes into self-dual and anti-self-dual parts ($\WW^{\pm}$). A. Polombo \cite{polombo92} derived a Bochner-Weitzenb\"ock formula for a Dirac operator and obtained applications for $\WW^{\pm}$. Then, M. Micallef and M. Wang \cite{mw93} proved that a closed four-dimensional manifold with harmonic self-dual and non-negative isotropic curvature (on the self-dual part) must be either conformally flat or quotient of a K\"{a}hler manifold with constant scalar curvature. As mentioned earlier, the curvature assumption is equivalent to the point-wise bound: for any eigenvalue $\omega$ of $\WW^+$, $$\omega \leq \frac{\SS}{6}.$$ 
Furthermore, Gursky \cite{g00} made use of the point-wise Bochner-Weitzenb\"ock  formula for Weyl tensor to derive optimal $L^2$-gap theorems. \\

In this paper, we derive analogous rigidity results for a closed Riemannian manifold with harmonic Weyl curvature in any higher dimension. First, we show a generalization of Tachibana's theorem. 
\begin{theorem}
\label{nonnegative}
Let $(M^n,g)$, $n\geq 5$, be a closed Riemannian manifold with harmonic Weyl curvature and non-negative curvature operator. Then $(M,g)$ must be either locally conformally flat or locally symmetric. 
\end{theorem} 
\begin{remark} In dimension four, the same statement holds and is a weaker version of Micallef and Wang's theorem. \end{remark}
\begin{remark} We also obtain other results in terms of pinching eigenvalues and norms, see Theorems \ref{mainpointwise} and \ref{normthm}. 
\end{remark}

Next, by using the solution to the Yamabe's problem, we obtain an integral gap.
\begin{theorem} 
\label{mainintegral}
Let $(M^n,g)$, $n\geq 6$, be a closed Riemannian manifold with harmonic Weyl tensor and positive Yamabe constant $\lambda(g)$. Then, 
\[a_1(n)||\WW||_{L^{n/2}}+a_2(n)||\EE||_{L^{n/2}}\geq \frac{n-2}{4(n-1)}{\lambda(g)},\]
unless $(M,g)$ is locally conformally flat. Here, for $\alpha$ the bigger root of the quadratic,
\[8(n-1)^2\alpha^2-2n(n-1)(n-2)\alpha+n(n-2)(n-3)=0,\] we have
\begin{align*}
a_1(n) &=\frac{10(n-1)\alpha^2}{2(n-1)\alpha-n+3} \text{,} &a_2(n) &=\frac{2(n-1)\alpha^2}{2(n-1)\alpha-n+3}\sqrt{\frac{n-1}{n}}. 
\end{align*}
\end{theorem} 
\begin{remark} For large $n$, $\alpha \approx n/4$ and $a_1\approx 5n/4$ and $a_2\approx n/4$. 
\end{remark}
\begin{remark} The cases $n=5$ and $n=4$ are treated somewhat differently, see Theorem \ref{case5} and Section \ref{fourdim} respectively.
\end{remark}
\begin{remark} When the manifold is locally conformally flat, there are rigidity results of the same flavor for constant scalar curvature and pinching condition on $|\EE|$; see \cite{grs07conformal, catino16conf} and references therein. 
\end{remark}
Note that, if one desires classification of Riemannian manifolds by diffeomorphism types, there are following rigidity results. Several authors, G. Huisken \cite{Huisken}, C. Margerin \cite{margerin86}, and S. Nishikawa \cite{nishi86}, independently showed that pointwise pinching on the Weyl and traceless Ricci curvature imply diffeomorphic to a space form. E. Hebey and M. Vaugon \cite{hv96} extended it to $L^{n/2}$ pinching. In dimension four, Margerin \cite{margerin98} was able to obtain optimal pointwise estimates. Furthermore, by using a solution to a fourth order fully nonlinear equation in conformal geometry, S. Chang, M. Gursky, and P. Yang \cite{CGY03} proved an $L^2$-sharp version.  

Putting everything together, the picture becomes clearer. If $\WW$ and $\EE$ are small, either in point-wise or in integral sense, then $(M,g)$ must be diffeomorphic to a space form. What is more, if it has harmonic Weyl curvature then isometrically $(M,g)$ must be conformally flat or locally symmetric.\\

The key ingredients in our proof are new Bochner-Weitzenb\"ock-Lichnerowicz (BWL) type identities connecting the gradient and the divergence of $\WW$. Here is a prototype.
\begin{theorem}
\label{main1}
For a closed manifold $(M^n,g)$, $n\geq 4$, we have: 
\begin{align}
\label{BWintegral}
\int_M |\nabla \WW|^2-\frac{(n-2)}{n-3}\int_M |\delta\WW|^2&=-2\int_M \left\langle{\WW^{\flat},\delta(\nabla \WW)-\nabla (\delta\WW)}\right\rangle \nonumber \\
&= 2\int_M \left\langle{\WW,\WW^2+\WW^\sharp}\right\rangle -\int_M\left\langle{\Rc\circ g, \WW^2}\right\rangle.
\end{align}
Here, $\WW^\flat_{iklj}=\WW_{jikl}$.
\end{theorem}
\begin{remark}
For relevant notation, see Section \ref{notation}.
\end{remark}
\begin{remark}
In dimension four, the result could be deduced from the work of Chang, Gursky, and Yang \cite[Equation 3.23]{CGY03}, see Section \ref{fourdim}. Also, A. Lichnerowicz \cite{lich58} derived an identity relating the gradient and the divergence of Riemann curvature, which is a vital step in the proof of Tachibana's theorem. 
\end{remark}
\begin{remark} For more identities, see Section \ref{gradvsdiv}.
\end{remark}

The idea is as follows. BWL identities relate zero-order terms (without derivatives) and derivative terms of some curvature quantity; see \cite[Chapter 7]{pe06book} or \cite{caotran1}. Then, appropriate conditions on curvature lead to a definite sign of the zero-order terms. That sign, however, might be opposite to the sign of the derivative terms, leading to some contradiction. Thus, generally our argument relies on new identities to conclude that $\nabla \WW\equiv 0$. Then, by a result of A. Derdzinski and W. Roter \cite[Theorem 2]{dr77}, parallel Weyl implies either locally conformally flat or locally symmetric. 
 

The organization of the paper is as follows. In Section \ref{notation}, we explain our convention and recall several preliminaries. Section \ref{gradvsdiv} discusses the connection between gradient and divergence of the Weyl tensor and proves several identities including Theorem \ref{main1}.  In Section \ref{gap}, we obtain rigidity theorems, pointwise and integral. 
Finally, the Appendix collects some useful computation and explain why dimension four is special.  
 
{\bf Acknowledgments:} We thank Professors Xiaodong Cao, Richard Schoen, and Bennett Chow for various discussion and suggestions. 

\section{\textbf{Preliminaries}}
\label{notation}
Throughout the paper, we adopt the following notation.
\begin{itemize}
\item $\{e_i\}_1^n$ is a normal orthonormal frame. 
\item $\wedge^k$ is the space of $k$-forms, or fully anti-symmetric $(k,0)$ tensor. 
\item $S^2(.)$ is the space of self-adjoint/symmetric operators. 
\item $S_0^2(.)$ is the space of traceless symmetric operators.
\end{itemize}
The Riemann curvature $(3,1)$ tensor and $(4,0)$ tensor are defined as follows,
\begin{align*}
\RR(X,Y,Z)&=-\nabla^2_{X,Y}Z +\nabla^2_{Y,X}Z\\
\RR(X,Y,Z,W)&=-\left\langle{\nabla^2_{X,Y}Z -\nabla^2_{Y,X}Z, W}\right\rangle_g.
\end{align*}
\begin{remark}
Our (3,1) curvature sign agrees with \cite{besse, brendlebook10}, but not \cite{chowluni, H3, pe06book}. Our (4,0) curvature convention is the same as \cite{besse, H3, brendlebook10}, but not \cite{chowluni, pe06book}.
\end{remark}
Consequently, we have the Ricci identity, for a $p$-tensor S,
\begin{align}
\label{ricciidentity}
&(\nabla^2_{X,Y} S)(Z_1,...,Z_p)-(\nabla^2_{Y,X} S)(Z_1,...,Z_p),\\
&=\sum_{i,k}\RR(X,Y,Z_i, e_k)S(...Z_{i-1},e_k, Z_{i+1}..).\nonumber
\end{align}

For any $(m,0)$-tensor $T$, its divergence is defined as
\[(\delta T)_{p_{2}...p_{m}}\doteqdot \sum_{i} \nabla_{i}T_{p_{2}...p_{m}i}. \]

 For $T\in S^2(\wedge^2)$, $|T|$ denotes the operator norm (which is $1/2$ of the tensor norm). Accordingly, the convention of norms is as follows. 
 \begin{itemize}
 \item If $T\in T^{*}\otimes S^2(\wedge^2)$ , $|T|^{2}=\sum_{i}\sum_{a<b; c<d}(T_{iabcd})^2. $
\item If $T \in \wedge^2 \otimes T^{*} $, $|T|^2=\sum_{i}\sum_{a<b}(T_{abi})^2. $
\item If $T\in \Lambda^3\otimes \Lambda^2$, $|T|^{2}=\sum_{i<j<k}\sum_{a<b}(T_{ijkab})^2.$
\end{itemize} 


\subsection{Algebraic curvature} 
\label{algC}
The exposition here follows from \cite[Chapter 1.G]{besse}. 

Let V be a real vector space equipped with a non-degenerate quadratic form $g$ and an orthonormal basis $\{e_i\}_{i=1}^n$. 
Associated with $g$ is an orthogonal group $O(g)$. 
The irreducible orthogonal decomposition of the $O(g)$-module $\otimes ^2 V$ is the following,
\[ \otimes^2 V= \wedge^2 V\oplus S_0^2(V) \oplus \mathbb{R}g.\]

We recall the Bianchi map on $\otimes^4 (V)$,
\[b(R)(X,Y,Z,W)\doteqdot \frac{1}{3}(R(X,Y,Z,W)+R(Y,Z,X,W)+R(Z,X,Y,W)).\]

Hence we have the $GL(V)$-equivariant decomposition,
\[S^2 (\wedge^2 V)=\text{Ker}b\oplus \text{Im}b.\] 

\begin{definition} $\mathfrak{C}V\doteqdot \text{Ker}(b)\cap S^2(\wedge^2 V)$ is the space of algebraic curvature tensor. 
\end{definition}

We recall the Ricci contraction, an $O(g)$-equivariant map $S^2 (\wedge^2 V)\mapsto S^2(V)$,
\[ rc(R)(X,Y)=\text{tr}_g R(X,.,Y,.).\]

Conversely, the Kulkarni-Nomizu product gives a canonical way to build an element of $S^2(\wedge^2 V)$ from two elements of $S^2(V)$.
\begin{align*}
(h\circ k) (X,Y,Z,W)= & ~h(X,Z)k(Y,W)+ k(X,Z)h(Y,W)\\
&-h(X,W)k(Y,Z)-k(X,W)h(Y,Z).\end{align*}

The map $g\circ .: S^2(V)\mapsto \mathfrak{C}V$ is injective and its adjoint is precisely $rc$. 
\begin{equation}
\label{selfadjoint}
 \left\langle{g\circ k, R}\right\rangle=\left\langle{k, rc(R)}\right\rangle .\end{equation}

If $n\geq 4$, the $O(g)$-module $\mathfrak{C}V$ has the following decomposition into unique irreducible subspaces, 
\begin{align*}
\mathfrak{C}V &= \mathfrak{S}V\oplus \mathfrak{Rc}V\oplus \mathfrak{W}V;\\
 \mathfrak{S}V &=\mathbb{R}g\circ g,\\
 \mathfrak{Rc}V &=q\circ (S_0^2 V),\\
 \mathfrak{W}V &=\text{Ker}(rc)\cap \text{Ker}b.
\end{align*}
Conveniently, $\mathfrak{S}V$ is just a multiple of the identity, $\mathfrak{Rc}V$ the traceless Ricci part, and $\mathfrak{W}V$ the Weyl curvature. By symmetry, $\mathfrak{W}V$ is trivial in dimension less than four.  

Next, we collect some definitions of quadratic product on algebraic curvature tensors (they are also related to Ricci flow, see, for example, \cite{H3, bohmwilking, brendlebook10}).
\begin{definition} For algebraic curvature tensors $R,S$, we define, 
\begin{align}
\label{Rsquared}
(R.S)_{ijkl}&=\frac{1}{2}\sum_{p,q}R_{ijpq}S_{klpq},\\
\label{Rsharp}
(R\sharp S)_{ijkl}&=\frac{1}{2}\sum_{p,q}\Big[R_{ipkq}S_{jplq}+S_{ipkq}R_{jplq}\\
&-R_{iplq}S_{jpkq}-S_{iplq}R_{jpkq}];\nonumber\\
R^2(X,Y,Z,W)&=(R.R)(X,Y,Z,W)=\frac{1}{2}\sum_{p,q}R(X,Y,e_{p},e_{q})R(Z,W,e_{p},e_{q})\nonumber \\
R^{\sharp}(X,Y,Z,W) &=(R\sharp R)(X,Y, Z,W)\nonumber\\
&=\sum_{p,q}(R(X,e_{p},Z,e_{q})R(Y,e_{p},W,e_{q})-R(X,e_{p},W,e_{q})R(Y,e_{p},Z,e_{q}))\nonumber.
\end{align}
\end{definition}
\begin{remark} It is clear that both products are distributive with respect to addition. Also $\sharp$ is commutative. 
\end{remark}
Even though $R^2$ and $R^\sharp$ are not necessarily in $\mathfrak{C}V$, $R^2+R^\sharp$ is. Furthermore, 
\begin{equation}
\label{rcsimplify}
rc(R^2+R^\sharp)(X,Y)=\sum_{p,q}R(X,e_p,Y,e_q)rc(R)(e_p, e_q).\end{equation}
Finally, the 3-linear map, 
\begin{equation}
\text{tri}(R_1, R_2, R_3)\doteqdot \left\langle{R_1.R_2+R_2.R_1+2R_1\sharp R_2,R_3}\right\rangle, 
\end{equation}
is symmetric in all three components $R_1, R_2, R_3\in S^2(\wedge^2 V)$. 
\subsection{Curvature Decomposition}
Using the algebraic disucssion above, the Riemann curvature has the following decomposition:
\begin{equation*}
\RR=\WW+\frac{\SS g\circ g }{2n(n-1)} + \frac{\EE \circ g}{n-2}=\WW-\frac{\SS g\circ g}{2(n-2)(n-1)} + \frac{\Rc \circ g}{n-2}.
\end{equation*}   
We note that, as $(4,0)$ tensors at a point $x$, $\WW\in \mathfrak{W}(T_xM)$, $\EE\circ g\in \mathfrak{Rc}(T_xM)$, and $\SS g\circ g\in \mathfrak{S}(T_xM)$ are orthogonal components. Consequently,
\begin{align*}
\RR &= \WW+\frac{\SS g\circ g }{2n(n-1)} + \frac{\EE \circ g}{n-2}=\WW+U+V, \\
|\RR|^2 &=|\WW|^2+|U|^2+|V|^2, \\
|U|^2 &=\frac{1}{2n(n-1)}\SS^2,\\
|V|^2 &=\frac{1}{n-2}|\EE|^2.
\end{align*}
The Weyl curvatutre is an algebraic curvature which is trace-less, hence, 
\[W_{1212} =\sum_{2<i<j}W_{ijij}.\] 
\begin{remark}
More generally, decompose the tangent space into orthogonal sub-spaces $N_1$ and $N_2$. Then for any orthonormal bases of $N_1$ and $N_2$, by the trace-less property, 
\begin{equation}
\label{weylsect}
\WW_{N_1}\doteqdot \sum_{i<j, i,j\in N_{1}}W_{ijij} =\sum_{k<l, k,l\in N_{2}}W_{klkl}\doteqdot \WW_{N_2}.\end{equation}
So the Weyl ``generalized sectional curvatures'' $\WW_{N_i}$ is independent of the choice of basis and, thus, well-defined. And those of complementing sub-spaces are equal. 
\end{remark}

In dimension four, the decomposition becomes,
\[\RR =\WW+\frac{\SS}{24} g\circ g+ \frac{1}{2}\EE \circ g.\]

A special feature of dimension four is that the Hodge $\ast$ operator decomposes the space of two-forms orthogonally according to eigenvalues $\pm1$: $\wedge^{2}=\wedge^2_{+}\otimes \wedge^2_{-}$. Also, we observe that $\WW(\Lambda^{2}_{\pm})\in \Lambda^{2}_{\pm}$; thus, it is unambiguous to define $\WW^{\pm}\doteqdot \WW^{|\Lambda_{\pm}}$

Accordingly, the curvature as an operator on $\wedge^2$ is, 
\begin{equation}
\label{curdec}
\RR=
 \left( \begin{array}{cc}
A^{+} & C \\
C^{T} & A^{-} \end{array} \right).
\end{equation}
Here, C is essentially the traceless Ricci. In addition, 
\begin{align*}
A^{\pm} &=\WW^{\pm}+\frac{S}{12}\text{Id}^{\pm},\\
|A^{\pm}|^{2} &=|\WW^{\pm}|^{2}+\frac{\SS^{2}}{48}, \\
|\Rc|^2-\frac{\SS^2}{4}&=|\EE|^2=4|C|^2=4\text{tr}(CC^T).
\end{align*}

\section{\textbf{Gradient Versus Divergence}}
\label{gradvsdiv}
This section investigates the connection between gradient and divergence of the Weyl curvature. To be compatible with the discussion of Subsection \ref{algC}, let $\mathfrak{X}(TM)$ be the space of all smooth sections such that at a point $x$, each section is an algebraic element of $\mathfrak{X}(T_{x}M)$. 

\subsection{Pointwise Identities} 
Recall the following operators, for $T\in \mathfrak{C}(TM)$, $A\in \Lambda^2\otimes T^{*}M$
\begin{align}
B(T)_{ijkmn}&=\nabla_i T_{jkmn}+\nabla_j T_{kimn}+\nabla_k T_{ijmn};\\
(A\circ' g)_{ijkmn}&=g_{kn}A_{ijm}+g_{in}A_{jkm}+g_{jn}A_{kim}\nonumber \\
&+g_{km}A_{jin}+g_{im}A_{kjn}+g_{jm}A_{ikn}.
\end{align}
\begin{remark} If $b$ corresponds to the first Bianchi identity, $B$ does the second. Also, \[B(T), A\circ' g\in \Lambda^3\otimes \Lambda^2.\] 
\end{remark} 
\begin{lemma}
Suppose $(M^n,g)$ is a Riemannian manifold and $A\in \Lambda^2\otimes T^{*}M$ such that for any orthonormal frame, 
\[\sum_{i}A_{iji}=0.\] 
Then, 
\[|A\circ' g|^2= (n-3)|A|^2\]
\end{lemma}
\begin{proof} The statement follows from a tedious computation. Below are the key steps. 
\begin{align*}
|A\circ' g|^2&=\frac{1}{12}\sum_{i,j,k,m,n}(A\circ' g)_{ijkmn}^2;\\
\sum_{i,j,k, m,n} (g_{kn}A_{ijm})^2&=n\sum_{i,j,m}A_{ijm}^2=2n|A|^2;\\
\sum_{i,j,k, m,n} g_{kn}g_{in}A_{ijm}A_{jkm}&=\sum_{j,m,n}A_{njm}A_{jnm}=-2|A|^2;\\ 
\sum_{i,j,k, m,n} g_{kn}g_{km}A_{ijm}A_{jin}&=\sum_{i,j,k}A_{ijk}A_{jik}=-2|A|^2;\\
\sum_{i,j,k, m,n} g_{kn}g_{im}A_{ijm}A_{kjn}&=\sum_{i,k,j}A_{iji}A_{kjk}=0.
\end{align*}
\end{proof}

Also, for a Riemann curvature, $B(\RR)=0$. That implies algebraic relations when applying map $B$ on components of the curvature. To see that, first we recall:
\begin{definition}
\label{framedef}
The tensors $P,Q: \wedge^2 TM \otimes T^*M$ are defined as:
\begin{align}
\label{equaP}
P_{ijk}&=\nabla_{i}\Rc_{jk}-\nabla_{j}\Rc_{ik}.\\
P(X\wedge Y,Z)&=(d_{\nabla}\Rc)(X,Y,Z)=\delta{\RR}(X,Y,Z), \nonumber \\
\label{equaQ}
Q_{ijk}&=g_{ki}\nabla_{j}\SS-g_{kj}\nabla_{i}\SS.
\end{align}
\end{definition}

\begin{remark} Thus, harmonic curvature is equivalent to the condition  that $\Rc$ is a Codazzi tensor.  
It is also well-known that, for example see \cite{chowluni},
\[ \delta \WW=-\frac{n-3}{n-2}(P+\frac{1}{2(n-1)}Q).
\]
\end{remark}

\begin{lemma} 
\label{bianchimap}
We have the followings:
\begin{enumerate}
\item $B(Rc\circ g)= P\circ' g;$
\item $B(S g\circ g)=-Q\circ' g ;$
\item $B(\WW)=\frac{1}{n-3}(\delta \WW) \circ' g.$
\end{enumerate}
\end{lemma}
\begin{proof} Since we consider equality of tensors, it suffices to show it for a normal orthonormal frame.
\begin{enumerate}
\item We compute: 
\begin{align*}
(\Rc\circ g)_{jkmn}&=\Rc_{jm}g_{kn}+\Rc_{kn}g_{jm}-\Rc_{jn}g_{km}-\Rc_{km}g_{jn},\\
\nabla_i (\Rc\circ g)_{jkmn}&=\nabla_i\Rc_{jm}g_{kn}+\nabla_i\Rc_{kn}g_{jm}-\nabla_i\Rc_{jn}g_{km}-\nabla_i\Rc_{km}g_{jn},\\
B(\Rc\circ g)_{ijkmn}&=\nabla_i\Rc_{jm}g_{kn}+\nabla_i\Rc_{kn}g_{jm}-\nabla_i\Rc_{jn}g_{km}-\nabla_i\Rc_{km}g_{jn},\\
&+\nabla_j\Rc_{km}g_{in}+\nabla_j\Rc_{in}g_{km}-\nabla_j\Rc_{kn}g_{im}-\nabla_j\Rc_{im}g_{kn}\\
&+\nabla_k\Rc_{im}g_{jn}+\nabla_k\Rc_{jn}g_{im}-\nabla_k\Rc_{in}g_{jm}-\nabla_k\Rc_{jm}g_{in}\\
&=g_{kn}P_{ijm}+g_{in}P_{jkm}+g_{jn}P_{kim}\\
&+g_{km}P_{jin}+g_{im}P_{kjn}+g_{jm}P_{ikn}. 
\end{align*}
\item We compute:
\begin{align*}
(\SS g\circ g)_{jkmn}&=2\SS g_{jm}g_{kn}-2\SS g_{jn}g_{km},\\
\nabla_i (\SS g\circ g)_{jkmn}&=2\nabla_i\SS (g_{jm}g_{kn}-g_{jn}g_{km}),\\
B(\SS g\circ g)_{ijkmn}&=2\nabla_i\SS (g_{jm}g_{kn}-g_{jn}g_{km}),\\
&+2\nabla_j\SS (g_{km}g_{in}-g_{kn}g_{im})\\
&+2\nabla_k\SS (g_{im}g_{jn}-g_{in}g_{jm})\\
&=g_{kn}Q_{jim}+g_{in}Q_{kjm}+g_{jn}Q_{ikm}\\
&+g_{km}P_{ijn}+g_{im}Q_{jkn}+g_{jm}Q_{kin}. 
\end{align*}
\item For the last statement, recall
\begin{align*}
B(\RR)&=0,\\
\RR &= \WW-\frac{S g\circ g}{2(n-2)(n-1)}+\frac{\Rc\circ g}{n-2},\\
(\delta \WW) &=-\frac{n-3}{n-2}(P+\frac{Q}{2(n-1)}).
\end{align*} 
\end{enumerate}
\end{proof}
\begin{remark} The key observation here is that harmonic Weyl curvature is equivalent to Weyl satisfying the second Bianchi identity.
\end{remark}
\begin{remark} 
Using the decomposition above, we can show that 
\[\frac{1}{n-3}|\delta \WW|^2=|B(W)|^2\leq 3|\nabla \WW|^2.\]
\end{remark}
The inequality above is, however, generally not sharp so we'll expose the connection between gradient and divergence by integration by parts.  
\subsection{Integral formula}
First, observe that by (\ref{rcsimplify}),
\[rc(\WW^2+\WW^\sharp)(.,.)=\sum_{p,q}\WW(.,e_p,.,e_q)rc(\WW)(e_p,e_q)\equiv 0.\]
Therefore, $\WW^2+\WW^\sharp\in \mathfrak{W}(TM)$. 

Let $A\in S^2(TM)$ and $B=A\circ g\in \mathfrak{C}(TM)$. We have the following. 
\begin{lemma}\label{productW}
We have, for a basis diagonalizing both $A$ and $g$ at a point then, 
\begin{align}
\label{pro1}
\left\langle{B, \WW^2+\WW^{\sharp}}\right\rangle &=0,\\
\left\langle{\WW^2, B}\right\rangle=-\left\langle{\WW, B\sharp \WW}\right\rangle &=\frac{1}{2}\sum_{i,j,p,q} A_{ii}\WW_{ijpq}^2,\\
\left\langle{\WW, \RR\sharp \WW}\right\rangle &= \frac{1}{2}\sum_{i,j,k,l, p, q}\WW_{ijkl}\WW_{jplq}\RR_{ipkq}. 
\end{align}
\end{lemma}
\begin{proof}
By (\ref{selfadjoint}), we have,  
\begin{align*}
\left\langle{B, \WW^2+\WW^{\sharp}}\right\rangle &=\left\langle{A, rc(\WW^2+\WW^{\sharp})}\right\rangle.
\end{align*}
Thus (\ref{pro1}) follows since $\WW^2+\WW^\sharp\in \mathfrak{W}(TM)$. 

Next, if we choose a basis diagonalizing both $A$ and $g$ then,
\begin{align*}
8\left\langle{B,\WW^2}\right\rangle &=\sum_{i,j,k,l} B_{ijkl}\sum_{p,q}\WW_{ijpq}\WW_{pqkl}\\
&=2\sum_{i,j,p,q} B_{ijij}\WW_{ijpq}\WW_{ijpq}=-8\left\langle{\WW, B\sharp \WW}\right\rangle.
\end{align*}
The last equality follows from the definition of $\sharp$ operator and symmetry of $\WW$ and $B$. Next, as $B=A\circ g$,
\begin{align*}
2\sum_{i,j,p,q} B_{ijij}\WW_{ijpq}\WW_{ijpq} &=2\sum_{i,j,p,q}(A_{ii}+A_{jj})\WW_{ijpq}\WW_{ijpq}\\
&=4\sum A_{ii}\WW_{ijpq}\WW_{ijpq}.
\end{align*}
Finally, by re-indexing, 
\begin{align*}
8\left\langle{\WW, \RR\sharp \WW}\right\rangle &= \sum_{i,j,k,l}\WW_{ijkl}\sum_{p,q}\Big[\RR_{ipkq}\WW_{jplq}+\WW_{ipkq}\RR_{jplq}\\
&-\RR_{iplq}\WW_{jpkq}-\RR_{jpkq}\WW_{iplq}\Big]\\
&=2\sum_{i,j,k,l,p,q}\WW_{ijkl}(\RR_{ipkq}\WW_{jplq}+\WW_{ipkq}\RR_{jplq});\\
&=4\sum_{i,j,k,l,p,q}\WW_{ijkl}\WW_{jplq}\RR_{ipkq}.
\end{align*}

\end{proof}
Inspired by Tachibana \cite{tachibana}, we compute the following.
\begin{lemma}
\label{tachilemma}
For a closed manifold $(M^n,g)$, $n\geq 4$, we have: 
\begin{equation*}
\left\langle{\WW^{\flat},\delta(\nabla \WW)-\nabla (\delta\WW)}\right\rangle =-\left\langle{\WW,\WW^2+\WW^\sharp}\right\rangle+\frac{1}{2}\left\langle{\Rc\circ g,\WW^2}\right\rangle.
\end{equation*}
Here, $\WW^\flat_{iklj}=\WW_{jikl}$.
\end{lemma}
\begin{proof}
We compute,
\begin{align*}
-4\left\langle{\WW^{\flat},\delta(\nabla \WW)-\nabla (\delta\WW)}\right\rangle &= \sum_{i,j,k,l}\WW_{ijkl}[(\delta(\nabla \WW))_{iklj}-(\nabla (\delta\WW))_{iklj}],\\
&=\sum_{i,j,k,l,m}\WW_{ijkl}(\nabla_m\nabla_i\WW_{jmkl}-\nabla_i\nabla_m\WW_{jmkl})
\end{align*}
Using (\ref{ricciidentity}), 
\begin{align*}
\nabla_{m}\nabla_i\WW_{jmkl}-\nabla_{i}\nabla_m\WW_{jmkl} &= \sum_s \RR_{mijs}\WW_{smkl}+\RR_{mims}\WW_{jskl}+\RR_{miks}\WW_{jmsl}+\RR_{mils}\WW_{jmks}.
\end{align*}
Using the Bianchi first identity and re-indexing, we have,
\begin{align*}
2\sum_{m,i,j,k,l,s}\WW_{ijkl}\RR_{mijs}\WW_{smkl}&=\sum_{m,i,j,k,l,s}\RR_{mijs}\WW_{ijkl}\WW_{smkl}+\RR_{mijs}\WW_{ijkl}\WW_{smkl}\\
&=\sum_{m,i,j,k,l,s}\RR_{mijs}\WW_{ijkl}\WW_{smkl}+\RR_{mjis}\WW_{jikl}\WW_{smkl}\\
&=\sum_{m,i,j,k,l,s}\RR_{mijs}\WW_{ijkl}\WW_{smkl}+\RR_{jmis}\WW_{ijkl}\WW_{smkl}\\
&=\sum_{m,i,j,k,l,s}(\RR_{mijs}+\RR_{jmis})\WW_{ijkl}\WW_{smkl}\\
&=-\sum_{m,i,j,k,l,s}\RR_{ijms}\WW_{ijkl}\WW_{smkl}\\
&=8\left\langle{\RR,\WW^2}\right\rangle.
\end{align*}
Next, if we choose a basis diagonalizing both $\Rc$ and $g$ then, by Lemma \ref{productW}, for $T=Rc\circ g$,
\begin{align*}
2\sum_{m,i,j,k,l,s}\WW_{ijkl}\RR_{mims}\WW_{jskl}&=2\sum_{i,j,k,l}\RR_{ii}\WW_{ijkl}\WW_{jikl}\\
&=-4\left\langle{T,\WW^2}\right\rangle.
\end{align*}
Next, also by Lemma \ref{productW},
\begin{align*}
2\sum_{m,i,j,k,l.s}\WW_{ijkl}(\RR_{miks}\WW_{jmsl}+\RR_{mils}\WW_{jmks})&=4\WW_{ijkl}\RR_{miks}\WW_{jmsl}\\
&=8\left\langle{\WW,\RR\sharp\WW}\right\rangle.
\end{align*}
Summing equations above yields,
\[-4\left\langle{\WW^{\flat},\delta(\nabla \WW)-\nabla (\delta\WW)}\right\rangle=4\left\langle{\RR,\WW^2}\right\rangle-2\left\langle{T,\WW^2}\right\rangle+4\left\langle{\WW,\RR\sharp\WW}\right\rangle.
\]
Recalling that
\[ \RR=\WW+\frac{\SS g\circ g }{2n(n-1)} + \frac{\EE \circ g}{n-2}:=\WW+B.\]

Since the curvature products are distributive, we obtain,
\begin{align*}
&\left\langle{\RR,\WW^2}\right\rangle+\left\langle{\WW,\RR\sharp \WW}\right\rangle \\
=&\left\langle{\WW,\WW^2}\right\rangle+\left\langle{B,\WW^2}\right\rangle+\left\langle{\WW,\WW\sharp \WW}\right\rangle+\left\langle{B,\WW^\sharp}\right\rangle \\
=& \left\langle{\WW,\WW^2+\WW^\sharp}\right\rangle.
\end{align*}

The last equality is due to Lemma \ref{productW}. So the result follows.
\end{proof}
Now, we are ready to prove Theorem \ref{main1}. The proof is inspired by some arguments from \cite[Proposition 3.3]{CGY03}.
\begin{proof} 
We have,
\begin{align*}
4\int_M |\nabla \WW|^2 &= \int_M \sum_{m, i,j,k,l} \nabla_m \WW_{ijkl}(B(W)_{mijkl}-\nabla_i \WW_{jmkl}-\nabla_j \WW_{mikl}),\\
\int_M \sum_{m, i,j, k,l} \nabla_m \WW_{ijkl}B(W)_{mijkl} &=\int_M \sum_{m, i,j, k,l} \nabla_m \WW_{ijkl}\frac{1}{n-3}(\delta \WW\circ' g)_{mijkl}\\
&=\frac{1}{n-3}\int_M \sum_{i,j, k,l}\nabla_m\WW_{ijkm}(\delta{\WW})_{ijk}+\nabla_m\WW_{ijml}(\delta\WW)_{jil}\\
&=\frac{4}{n-3}\int_M |\delta\WW|^2,\\
-2\int_M\sum_{m,i,j, k,l} \nabla_m \WW_{ijkl}\nabla_i \WW_{jmkl}&=-2\int_M\sum_{m,i,j, k,l} \nabla_m \WW_{ijkl}\nabla_j \WW_{mikl}\\
&=2\int_M\sum_{m,i,j, k,l}\WW_{ijkl}\nabla_m\nabla_i\WW_{jmkl}\\
&= -8\int_M \left\langle{\WW^{\flat}, \delta(\nabla \WW)-\nabla(\delta\WW)}\right\rangle+2\int_M\sum_{m,i,j, k,l}\WW_{ijkl}\nabla_{i}\nabla_m\WW_{jmkl};
\end{align*}
Using integration by parts yields, 
\begin{align*}
2\int_M\sum_{m,i,j, k,l}\WW_{ijkl}\nabla_{i}\nabla_m\WW_{jmkl} &=-2\int_M\sum_{m,i,j, k,l}\nabla_{i}\WW_{ijkl}\nabla_m\WW_{jmkl}=4\int_M |\delta\WW|^2.
\end{align*}
Summing equations above and recalling Lemma \ref{tachilemma} yield, 
\begin{align*}
4\int_M |\nabla \WW|^2 &=\frac{4(n-2)}{n-3}\int_M |\delta\WW|^2\\
&+8\int_M \left\langle{\WW,\WW^2+\WW^\sharp}\right\rangle-4\int_M \left\langle{\Rc\circ g,\WW^2}\right\rangle\nonumber
\end{align*}

Then the theorem follows.
\end{proof}

\subsection{Harmonic Weyl Curvature}
In this section, using the method similar to above, we obtain identities for a manifold with harmonic Weyl curvature. First, we show the following identity for the Ricci tensor.
\begin{proposition}
\label{ricci}
For a closed manifold $(M^n,g)$ with harmonic Weyl tensor, $n\geq 4$, 
\begin{equation}
\int_M |\nabla \Rc|^2-\frac{n}{4(n-1)}\int_M |\nabla \SS|^2=\int_M \WW(\EE,\EE)-\frac{n}{n-2}\EE^3-\frac{1}{n-1}\SS|\EE|^2.
\end{equation}
Here, for an orthonormal frame, $\WW(\EE,\EE)=\sum_{i,j,k,l}\WW_{ijkl}\EE_{ik}\EE_{jl}$, $\EE^3=\sum_{i,j,k}\EE_{ij}\EE_{jk}\EE_{ki}$.
\end{proposition}
\begin{proof}
First, we recall, by the contracted second Bianchi identity,
\[ \delta\Rc=\frac{1}{2}\nabla \SS.\]
Also, since $\delta \WW=0$,
\[-P=\frac{1}{2(n-1)}Q.
\]
We have,
\begin{align*}
\int_M |\nabla \Rc|^2 &= \int_M \sum_{i,j,k} \nabla_i \RR_{jk}(P_{ijk}+\nabla_j \RR_{ik}),\\
\int_M \sum_{i,j, k} \nabla_i \RR_{jk}P_{ijk} &=-\frac{1}{2(n-1)}\int_M \sum_{i,j, k} \nabla_i \RR_{jk}Q_{ijk}\\
&=-\frac{1}{2(n-1)}\int_M \sum_{i,j, k}\nabla_i\RR_{jk}(g_{ki}\nabla_j \SS-g_{kj}\nabla_i \SS)\\
&=\frac{1}{4(n-1)}\int_M |\nabla\SS|^2;\\
\int_M\sum_{i,j, k} \nabla_i \RR_{jk}\nabla_j \RR_{ik}&=-\int_M\sum_{i,j, k} \RR_{jk}\nabla_i\nabla_j \RR_{ik}.
\end{align*}
Using (\ref{ricciidentity}), 
\begin{align*}
\nabla_{i}\nabla_j\RR_{ik}-\nabla_{j}\nabla_i\RR_{ik} &= \sum_s \RR_{ijis}\RR_{sk}+\RR_{ijks}\RR_{is};\\
-\int_M\sum_{i,j,k}\RR_{jk}\nabla_{j}\nabla_i\RR_{ik} &=\int_M\sum_{i,j, k}\nabla_{j}\RR_{jk}\nabla_i\RR_{ik}=\frac{1}{4}\int_M |\nabla\SS|^2.
\end{align*}
For the zero order terms we compute, using the Bianchi first identity and re-indexing,
\begin{align*}
-\sum_{i,j,k,s}\RR_{jk}\RR_{ijis}\RR_{sk}&=-\sum_{s,j,k}\RR_{jk}\RR_{js}\RR_{sk};\\
-\sum_{i,j,k,s}\RR_{jk}\RR_{ijks}\RR_{is}&=\sum_{i,j,k,s}\RR_{ijks}\RR_{ik}\RR_{js}.\end{align*}
By a tedious computation, 
\begin{align*}
\Rc^3\doteqdot \sum_{s,j,k}\RR_{jk}\RR_{js}\RR_{sk} &=\EE^3+\frac{3}{n}\SS|\Rc|^2-\frac{2\SS^3}{n^2}\\
&=\EE^3+\frac{3}{n}\SS|\EE|^2+\frac{\SS^3}{n^2};\\
\RR(\Rc,\Rc)\doteqdot \sum_{i,j,k,s}\RR_{ijks}\RR_{ik}\RR_{js}&=\WW(\Rc,\Rc)-\frac{2}{n-2}\EE^3+\frac{\SS^3}{n^2}+\frac{2n-3}{n(n-1)}\SS|\EE|^2.
\end{align*}
\end{proof}
\begin{remark}
Huisken \cite[Lemma 4.3]{Huisken} shows that for any Riemannian manifold, $|\nabla\Rc|^2\geq \frac{3n-2}{2(n-1)(n+2)}|\nabla \SS|^2$. 
\end{remark}
Furthermore, we have a point-wise identity for Weyl tensor. This is a generalization of one in dimension four, see \cite[16.73]{besse}.
\begin{theorem}
\label{pointwiseBW}
For a closed manifold $(M^n,g)$, $n\geq 4$, with harmonic Weyl curvature, we have: 
\begin{equation}
\label{pointwiseW}
\Delta |\WW|^2=2|\nabla \WW|^2-4\left\langle{\WW,\WW^2+\WW^\sharp}\right\rangle +2\left\langle{\Rc\circ g, \WW^2}\right\rangle.
\end{equation}
\end{theorem}
\begin{proof} 
We have, by Lemma \ref{bianchimap} and harmonic Weyl curvature, $B(\WW)\equiv 0$. Then,
\begin{align*}
2\Delta |\WW|^2&= 4|\nabla \WW|^2+\sum_{m,i,j,k,l}\WW_{ijkl}\nabla_m\nabla_m \WW_{ijkl};\\
\sum_{m,i,j,k,l}\WW_{ijkl}\nabla_m \nabla_m\WW_{ijkl} &=\sum_{m, i,j,k,l}  \WW_{ijkl}\nabla_m(-\nabla_i \WW_{jmkl}-\nabla_j \WW_{mikl}),\\
&=-2\sum_{m,i,j, k,l} \WW_{ijkl}\nabla_m\nabla_i \WW_{jmkl},\\
&= 8\left\langle{\WW^b, \delta(\nabla \WW)-\nabla(\delta\WW)}\right\rangle-2\sum_{m,i,j,k,l}\WW_{ijkl}\nabla_i(\delta \WW)_{klj},\\
&=8\left\langle{\WW^b, \delta(\nabla \WW)-\nabla(\delta\WW)}\right\rangle\\
&= -8\left\langle{\WW,\WW^2+\WW^\sharp}\right\rangle+4\left\langle{\Rc\circ g,\WW^2}\right\rangle.
\end{align*}
The last equality follows from Lemma \ref{tachilemma}. 

\end{proof}
\section{\textbf{Rigidity Theorems}}
\label{gap}
In this section, we derive applications of new Bochner-Weitzenb\"ock-Lichnerowicz type formulas. Precisely, results include point-wise and integral rigidity theorems.  
\subsection{Estimates}
First, we'll investigate the quantity, $\left\langle{\WW, \WW^2+\WW^\sharp}\right\rangle$. Recall that, by re-indexing,
\begin{align*}
8\left\langle{\WW, \WW^\sharp}\right\rangle &= 2\sum_{i,j,k,l,p,q}\WW_{ijkl}(\WW_{ipkq}\WW_{jplq}-\WW_{iplq}\WW_{jpkq}),\\
&= 4\sum_{i,j,k,l,p,q}\WW_{ijkl}\WW_{ipkq}\WW_{jplq};\\
8\left\langle{\WW, \WW^2}\right\rangle &=\sum_{i,j,k,l,p,q}\WW_{ijkl}\WW_{ijpq}\WW_{klpq}.
\end{align*}

When $4\leq n\leq 5$, due to \cite[A1]{jp87}, (be aware of a typo in (A1))
\begin{equation}\label{sharpcubic}
\left\langle{\WW, \WW^\sharp}\right\rangle=2\left\langle{\WW,\WW^2}\right\rangle\doteqdot 2\WW^3.
\end{equation}
Thus,
\begin{equation}
\label{five}
\left\langle{\WW,\WW^2+\WW^\sharp}\right\rangle= 3\WW^3. \end{equation}
\begin{remark} Subsection \ref{pure} compares $\left\langle{\WW, \WW^\sharp}\right\rangle$ with $\left\langle{\WW,\WW^2}\right\rangle$ in any dimension when the curvature is pure. 
\end{remark}
For a general dimension, we follow the argument in \cite{Huisken,tachibana} and define, for fixed m, n, p, q, a local skew symmetric tensor,
\begin{align*}
u_{ij}^{(mnpq)} &= \WW_{inpq}g_{jm}+\WW_{mipq}g_{jn}+\WW_{mniq}g_{jp}+\WW_{mnpi}g_{jq}\\
&-\WW_{jnpq}g_{im}-\WW_{mjpq}g_{in}-\WW_{mnjq}g_{ip}-\WW_{mnpj}g_{iq}.
\end{align*}
It follows that,
\begin{align*}
8|u|^2=\sum_{m,n,p,q}\left\langle{u_{ij}^{(mnpq)},u_{ij}^{(mnpq)}}\right\rangle &= 32(n-1)|\WW|^2,\\
8\left\langle{\WW,\WW^2+\WW^\sharp}\right\rangle &=\frac{1}{8}\left\langle{W_{ijkl} u_{ij}^{(mnpq)},u_{kl}^{(mnpq)}}\right\rangle.
\end{align*}
If $\omega$ is the largest absolute value of any eigenvalue of $\WW$, then by Berger's estimate \cite{berger60b},
\[\WW_{ijkl}\leq \frac{4}{3}\omega.\]
It follows that,
\begin{equation}
\label{six}
\left\langle{\WW,\WW^2+\WW^\sharp}\right\rangle =\frac{1}{64}\left\langle{W_{ijkl} u_{ij}^{(mnpq)},u_{kl}^{(mnpq)}}\right\rangle \leq \frac{4}{3}\frac{1}{64}\omega 8|u|^2= \frac{2}{3}(n-1)|\WW|^2 \omega.
\end{equation}
Thus, the computation above proves the following result. 
\begin{lemma}
\label{atleast5}
Let $\omega$ be the largest absolute value of any eigenvalue of $\WW$ then, for $n\geq 5$,
\begin{equation}
\left\langle{\WW,\WW^2+\WW^\sharp}\right\rangle\leq \frac{2(n-1)}{3} \omega|\WW|^2.
\end{equation}
\end{lemma}
\begin{remark} \label{fivecase} For $n=5$, the result holds with the same constant for largest eigenvalue of $\WW$ due to (\ref{five}) and Lemma \ref{Wcubic}. 
\end{remark}

We also have an estimate just in terms of norm. First, we have the following result, for a proof see \cite[Lemma 2.4]{Huisken}.
\begin{lemma} \label{eigenestimate}
Let $T$ be a symmetric trace-free operator on an $m$-dimensional vector space $V$. If $\lambda$ is an eigenvalue, then,
\[\lambda^2 \leq \frac{m-1}{m} |T|^2.\]
\end{lemma} 
Then the following holds.
\begin{lemma}
\label{atleast5norm}
For $n\geq 5$,
\begin{equation}
\left\langle{\WW,\WW^2+\WW^\sharp}\right\rangle\leq c(n)|\WW|^3.
\end{equation}
Here, we find $c(5)=\frac{8}{\sqrt{10}}$ and $c(n)=5$ for $n\geq 6$. 
\end{lemma}
\begin{proof}
For any dimension, by Cauchy-Schwartz inequality, 
\begin{align*}
\left\langle{\WW,\WW^2+\WW^\sharp}\right\rangle &=\frac{1}{8}\sum_{i,j,k,l,p,q}(4\WW_{ijkl}\WW_{ipkq}\WW_{jplq}+\WW_{ijkl}\WW_{ijpq}\WW_{klpq}),\\
&\leq 5|\WW|^3. 
\end{align*}
In dimension five, $\WW$ is an operator on a $10$-dimensional vector space. By Lemma \ref{Wcubic} and equation (\ref{eigenestimate}), 
\[\left\langle{\WW,\WW^2+\WW^\sharp}\right\rangle=3\WW^3\leq 3\sqrt{\frac{9}{10}}\frac{8}{9}|\WW|^3=\frac{8}{\sqrt{10}}|\WW|^3.\]
\end{proof}
\begin{remark} For $n>5$, $c(n)$ might be improved slightly, see \cite[Lemma 2.3]{catino15integral}. 
\end{remark}

\subsection{Point-wise}
In this subsection, we'll show various results under point-wise assumptions. 
First, we prove Theorem \ref{nonnegative}.

\begin{proof}
By the work of Tachibana \cite{tachibana} (see also \cite[Lemma 7.33]{chowluni}), non-negative curvature operator implies that,
\begin{align*}
\left\langle{\WW^{\flat},\delta(\nabla \WW)-\nabla (\delta\WW)}\right\rangle &\geq 0.
\end{align*}
Therefore, by Theorem \ref{main1} and harmonicity of Weyl, $\nabla\WW\equiv 0$. Then due to Derdzinski and Roter \cite[Theorem 2]{dr77}, the metric must be either locally conformally flat or locally symmetric. 
\end{proof}
The next result is concerned with bounding eigenvalues by the scalar curvature.

\begin{proposition} 
\label{p1} Let $(M^n,g)$, $n\geq 5$, be a closed Riemannian manifold. Let $\omega$ be the largest absolute value of any eigenvalue of $\WW$ and $-\ell$ the smallest eigenvalue of $\EE$. Suppose at each point,
$$\frac{2(n-1)}{3}\omega+\ell\leq \frac{\SS}{n},$$ then the following holds,
\[\int_M |\nabla \WW|^2\leq \frac{(n-2)}{n-3}\int_M (\delta\WW)^2.\]
In dimension five, it suffices to let $\omega$ be the largest eigenvalue of $\WW$.
\end{proposition}

\begin{proof}
By Theorem \ref{main1}, 
\begin{align*}
\int_M |\nabla \WW|^2 -\frac{(n-2)}{n-3}\int_M (\delta\WW)^2 &=2\int_M \left\langle{\WW,\WW^2+\WW^\sharp}\right\rangle -\int_M\left\langle{{\Rc\circ g},\WW^2}\right\rangle,\\
\int_M \left\langle{\WW^2,\Rc\circ g}\right\rangle &=\left\langle{(\EE+\frac{\SS}{n}g)\circ g,\WW^2}\right\rangle,\\
&=\left\langle{\EE\circ g,\WW^2}\right\rangle+\frac{2\SS}{n}|\WW|^2.
\end{align*}
By Lemma \ref{atleast5}, 
\[\left\langle{\WW,\WW^2+\WW^\sharp}\right\rangle\leq \frac{2(n-1)}{3}\omega|\WW|^2.\]
By standard inequality, an eigenvalue of $E\circ g$ must be greater than or equal to $-2\ell$. Thus, 
\[\left\langle{\EE\circ g,\WW^2}\right\rangle\geq -2\ell |\WW|^2.\]
Therefore, if $$\frac{4(n-1)}{3}\omega+2\ell\leq \frac{2\SS}{n},$$ 
then,
\[\int_M |\nabla \WW|^2 -\frac{(n-2)}{n-3}\int_M (\delta\WW)^2\leq 0.\]

In dimension five, we could use Remark (\ref{fivecase}) instead of Lemma \ref{atleast5}.
\end{proof}

Restricted to the case of harmonic Weyl curvature, we obtain the following result.
\begin{theorem}
\label{mainpointwise} Let $(M^n,g)$, $n\geq 5$, be a closed Riemannian manifold with harmonic Weyl tensor. Let $\omega$ be the largest absolute value of any eigenvalue of $\WW$ and $-\ell$ the smallest eigenvalue of $\EE$. Suppose at each point,
$$\frac{2(n-1)}{3}\omega+\ell\leq \frac{\SS}{n},$$ then $(M,g)$ must be either conformally flat or locally symmetric. In dimension five, it suffices to let $\omega$ be the largest eigenvalue of $\WW$.
\end{theorem}
\begin{proof}
For harmonic Weyl tensor, $\delta \WW\equiv 0$. Then, Prop.\ref{p1} implies that $\nabla \WW\equiv 0$. Thus, $(M,g)$ has parallel Weyl curvature. By Derdzinski and Roter \cite[Theorem 2]{dr77}, the metric must be either locally conformally flat or locally symmetric.
\end{proof}
The next theorem yields a similar conclusion under an assumption on the norm.
\begin{theorem}\label{normthm}
Let $(M^n,g)$, $n\geq 5$, be a closed Riemannian manifold. Suppose at each point,
$$c(n)|\WW|+\sqrt{\frac{n-1}{n}}|\EE|\leq \frac{\SS}{n},$$ then the following holds,
\[\int_M |\nabla \WW|^2\leq \frac{(n-2)}{n-3}\int_M (\delta\WW)^2.\]
Here, $c(n)$ is the same as in Lemma \ref{atleast5norm}.
\end{theorem}
\begin{proof}
The proof is similar to one above. The only difference is to obtain the following inequalities in terms of norm.
\begin{align*}
\left\langle{\WW,\WW^2+\WW^\sharp}\right\rangle &\leq c(n)|\WW|^3;\\
-\left\langle{\EE\circ g,\WW^2}\right\rangle &\leq 2\ell |\WW|^2\leq 2\sqrt{\frac{n-1}{n}}|\EE||\WW|^2.
\end{align*}
The first line is just Lemma \ref{atleast5norm} while the second line is justified by Lemma \ref{eigenestimate}.
\end{proof}

\subsection{Integral}
This subsection will focus on integral gap results using the solution to the Yamabe problem.

\textbf{Yamabe problem}: Given a compact Riemannian manifold $(M,g)$ of dimension $n\geq 3$, find a conformal metric with constant scalar curvature. 

For a conformal change of metric, the volume form and the scalar curvature transform as follows.
\begin{align*}
\overline{g}&=u^{4/(n-2)}g,\\
d\overline{\mu}&=u^{2n/(n-2)}d\mu,\\
\overline{\SS}&=u^{-(n+2)/(n-2)}\Big(-\frac{4(n-1)}{n-2}\Delta +\SS\Big)u=u^{-(n+2)/(n-2)}L_gu,\\
L_g&\doteqdot-\frac{4(n-1)}{n-2}\Delta +\SS. 
\end{align*}
Yamabe \cite{yam60} observed that the problem is equivalent to finding the minimizer of the functional,
\[Q[\overline{g}]=\frac{\int_M \overline{\SS} d\overline{\mu}}{(\int_M d\overline{\mu})^{(n-2)/n}}=Y[u]=\frac{\int_{M} u Lu d\mu}{(\int_M u^{2n/(n-2)}d\mu)^{(n-2)/n}}.\] 
The Yamabe constant is defined accordingly, 
\[\lambda(g)=\inf\{Y[u]:~ u>0,~ u \in C^2(M)\}.\]
The problem has an extensive history and was finally settled by R. Schoen using the positive mass theorem \cite{schoen84}; see \cite{yam60, tru68, aub76} for partial results and \cite{lp87} for an expository account.
When the Yamabe constant is positive, a consequence of the solution is a conformal Sobolev inequality, for $s_n = \frac{n-2}{4(n-1)}$,
\begin{equation}
\label{sobolev}
s_n\lambda (g)(\int_M |u|^{\frac{2n}{n-2}}d\mu_g)^{\frac{n-2}{n}}\leq \int_M|\nabla_g u|^2 d\mu_g+s_n\int_M\SS |u|^2 d\mu_g.  \end{equation}
This inequality will play an crucial role in the proof of Theorem \ref{mainintegral}. 

\begin{proposition} 
\label{integralrigid}
Let $(M^n,g)$, $n\geq 5$, be a closed Riemannian manifold with positive Yamabe constant $\lambda(g)$. Assume that, 
\[c_1(n)||\WW||_{L^{n/2}}+c_2(n)||\EE||_{L^{n/2}}\leq d(n){\lambda(g)}.\]
Then the following holds:
\[(1-\frac{d(n)}{s_n})\int_M |\nabla \WW|^2 \leq \frac{(n-2)}{n-3}\int_M |\delta\WW|^2+(d(n)-\frac{2}{n})\int_M \SS|\WW|^2.\]
Here, $c(n)$ is the same as in Lemma \ref{atleast5norm},
\begin{align*}
c_1(n) &=2c(n), &c_2(n) &=2\sqrt{\frac{n-1}{n}}.
\end{align*}
\end{proposition}
\begin{proof}
From Theorem \ref{main1}, 
\begin{align*}
\int_M |\nabla \WW|^2-\frac{(n-2)}{n-3}\int_M |\delta\WW|^2 &=2\int_M \left\langle{\WW,\WW^2+\WW^\sharp}\right\rangle -\int_M\left\langle{{\Rc\circ g},\WW^2}\right\rangle.
\end{align*}
By Lemma \ref{atleast5norm},
\begin{align*}
2\left\langle{\WW,\WW^2+\WW^\sharp}\right\rangle &\leq 2c(n)|\WW|^3= c_1(n)|\WW|^3.\\
-\left\langle{\WW^2, E \circ g}\right\rangle &\leq  2\ell|\WW|^2\leq  2\sqrt{\frac{n-1}{n}}|\EE||\WW|^2= c_2(n)|\EE||\WW|^2.
\end{align*}
Applying the H\"{o}lder's inequality and the conformal Sobolev's inequality, 
\begin{align*}
\int_M |\WW|^3 =||\WW||_{L^3}^3 &\leq ||\WW||_{L^{n/2}}||\WW||^2_{L^{2n/(n-2)}}\\
 &\leq \frac{||\WW||_{L^{n/2}}}{\lambda(g)}(\frac{1}{s_n}\int_M|\nabla |\WW||^2+\int_M \SS|\WW|^2);\\
\int_M |\EE||\WW|^2  &\leq ||\EE||_{L^{n/2}}||\WW||^2_{L^{2n/(n-2)}}\\
 & \leq \frac{||\EE||_{L^{n/2}}}{\lambda(g)}(\frac{1}{s_n}\int_M|\nabla |\WW||^2+\int_M \SS|\WW|^2).   
\end{align*}
Also, we recall the classical Kato inequality: 
$$|\nabla \WW|^2\geq |\nabla |\WW||^2.$$
Combining inequalities above yield,
\begin{align*}
\int_M|\nabla \WW|^2+ \frac{2\SS}{n}|\WW|^2 &\leq c_1(n)\int_M |\WW|^3+c_2(n)\int_M |\EE||\WW|^2,\\
&\leq \frac{c_1(n)||\WW||_{L^{n/2}}+c_2(n)||\EE||_{L^{n/2}}}{\lambda(g)}\Big(\frac{1}{s_n}\int_M |\nabla |\WW||^2+\int_M \SS|\WW|^2\Big),\\
&\leq d(n)(\frac{1}{s_n}\int_M |\nabla \WW|^2+\int_M \SS|\WW|^2).
\end{align*}

\end{proof}
The proof of Theorem \ref{mainintegral} is a refinement of the above.
\begin{proof}
For harmonic Weyl tensor, $\delta \WW\equiv 0$. Here, there is an improved Kato inequality for harmonic Weyl tensor \cite{branson00}, 
$$|\nabla \WW|^2\geq \frac{n+1}{n-1}|\nabla |\WW||^2 .$$
By Theorem \ref{pointwiseBW},
\begin{align*}
\Delta |\WW|^2 &=2|\nabla \WW|^2-4\left\langle{\WW,\WW^2+\WW^\sharp}\right\rangle +2\left\langle{\Rc\circ g, \WW^2}\right\rangle;\\
|\WW|\Delta|\WW| &=\frac{1}{2}\Delta |\WW|^2-|\nabla |\WW||^2\\
&\geq \frac{2}{n-1}|\nabla |\WW||^2-2\left\langle{\WW,\WW^2+\WW^\sharp}\right\rangle +\left\langle{\Rc\circ g, \WW^2}\right\rangle,\\
&\geq \frac{2}{n-1}|\nabla |\WW||^2+\frac{2\SS}{n}|\WW|^2-c_1(n)|\WW|^3-c_2(n)|\EE||\WW|^2.
\end{align*}
Here, $c_1(n), c_2(n)$ are as in Proposition \ref{integralrigid}. Now we suppose $|\WW|>0$ everywhere and let $u=|\WW|$ (if $\WW$ assumes value zero at some point, we could replace $u=|\WW|+\epsilon$, $\epsilon>0$ and use a standard limit argument for $\epsilon\rightarrow 0$). Then,
\begin{align*}
u^\alpha\Delta u^\alpha &=\frac{\alpha-1}{\alpha}|\nabla u^\alpha|^2+\alpha u^{2\alpha-2} u\Delta u,\\
&\geq (1-\frac{n-3}{(n-1)\alpha})|\nabla u^\alpha|^2+\frac{2\alpha}{n}\SS u^{2\alpha}-\alpha c_1(n)u^{2\alpha+1}-c_2(n)\alpha |\EE|u^{2\alpha} 
\end{align*}
Using integration by parts and applying H\"{o}lder's inequality yield,
\begin{align*}
0 &\geq (2-\frac{n-3}{(n-1)\alpha})\int_M|\nabla u^\alpha|^2+\frac{2\alpha}{n}\int_M\SS u^{2\alpha}-\alpha c_1(n)\int_M u^{2\alpha+1}-\alpha c_2(n) \int_M|\EE|u^{2\alpha},\\
&\geq (2-\frac{n-3}{(n-1)\alpha})||\nabla u^\alpha||_{L^2}^2+ \frac{2\alpha}{n}||\SS u^{2\alpha}||_{L^1}\\
&-\alpha c_1(n)||\WW||_{L^{n/2}} ||u^{\alpha}||_{L^{2n/(n-2)}}^2-\alpha c_2(n)||\EE||_{L^{n/2}} ||u^{\alpha}||_{L^{2n/(n-2)}}^2
\end{align*} 
Applying the Sobolev inequality (\ref{sobolev}), we have,
\begin{align*}
 & A(n,\alpha) \doteqdot (2-\frac{n-3}{(n-1)\alpha}); \\
& \alpha c_1(n)||\WW||_{L^{n/2}} ||u^{\alpha}||_{L^{2n/(n-2)}}^2+\alpha c_2(n)||\EE||_{L^{n/2}} ||u^{\alpha}||_{L^{2n/(n-2)}}^2 \\
&\geq A(n,\alpha) \int_M|\nabla u^\alpha|^2+\frac{2\alpha}{n}\int_M\SS u^{2\alpha}\\
&\geq s_n\lambda(g)A(n,\alpha)||u^{\alpha}||_{L^{2n/(n-2)}}^2+(\frac{2\alpha}{n}-A(n,\alpha)s_n)\int_M\SS u^{2\alpha}.
\end{align*}
We consider,
\begin{align*}
\frac{2\alpha}{n}-A(n,\alpha)s_n &=\frac{8(n-1)^2\alpha^2-2n(n-1)(n-2)\alpha+n(n-2)(n-3)}{4n(n-1)^2\alpha}.
\end{align*}
It follows that if $n\neq 5$, we could choose $\frac{n-3}{2(n-1)}<\alpha$ such that $\frac{2\alpha}{n}-A(n,\alpha)s_n=0$. In that case, if the assumption in the theorem does not hold, then $\WW\equiv 0$. 
\end{proof}

Note that, from the proof, if $n=5$ then the quadratic above is always positive. Therefore, it remains to choose $\frac{1}{4}<\alpha$ to minimize $\frac{2\alpha^2}{4\alpha-1}$. Thus, we choose $\alpha=\frac{1}{2}$ and use Lemma \ref{atleast5norm} to obtain the following result.
\begin{theorem}\label{case5}
Let $(M^5,g)$ be a closed Riemannian manifold with harmonic Weyl tensor and positive scalar curvature and Yamabe constant $\lambda(g)$. Then, unless $(M,g)$ is locally conformally flat, 
\[\frac{8}{\sqrt{10}}||\WW||_{L^{n/2}}+\frac{2}{\sqrt{5}}||\EE||_{L^{n/2}}\geq \frac{3}{16}{\lambda(g)}.\]
\end{theorem}
 
\section{\textbf{Appendix}}
First, we'll show a general estimate.
\begin{lemma}
\label{Wcubic}
Let $\sum_{i=1}^n x_i=0$ and $x_i\leq s$ then, 
\begin{equation}
x^3=\sum_{i=1}^n x_i^3\leq \frac{s(n-2)}{n-1} \sum_{i=1}^n x_i^2=\frac{s(n-2)}{n-1}x^2.
\end{equation}
\end{lemma}
\begin{proof}
When $x^2=\sum_{i=1}^n x_i^2=0$, the statement follows vacuously. 

When $x^2>0$, the problem reduces to maximizing $f(x)=\frac{x^3}{x^2}$ given constraints 
\begin{itemize}
\item $\sum_{i=1}^n x_i=0$;
\item $x_i\leq s>0$.
\end{itemize} 

By compactness of the constraints, the maximum exists. 
Also, we observe that the maximum must be at least
\begin{equation}\label{atleast}\frac{s^3-(n-1)(\frac{s}{n-1})^3}{s^2+(n-1)(\frac{s}{n-1})^2}=\frac{s(n-2)}{n-1}.
\end{equation}
By Karush-Kuhn-Tucker extension of the Lagrange multiplier method, a maximum of $f(x)$ is obtained at $x$ such that:
\begin{itemize}
\item $\nabla f(x)=\sum_{i=1}^n \mu_i \nabla (x_i-s)+\lambda \nabla (\sum_{i=1}^n x_i)$.
\item $\mu_i\geq 0$ and $\mu_i=0$ if $x_i<s$.
\end{itemize}
The first condition implies that, for any $i$,
\[3x_i^2 x^2-2x_i x^3=x^4 (\mu_i+\lambda). 
\]
By the second condition, if $x_i<s$, $x_i$ can be one of the roots, $y_1\leq  y_2$, of the quadratic on $y$: $3x^2 y^2-2x^3 y_i-\lambda x^4=0$. Therefore, there are two cases.

\textbf{Case 1:} There are $m$ $y_1$'s, $p$ $y_2$'s, and $q$ of $s$'s and $y_2<s$. Then we have,
\begin{align*}
my_1+py_2+qs &=0;\\
f(x)=\frac{x^3}{x^2} &=\frac{p y_2^3+q s^3+my_1^3}{py_2^2 +qs^2+my_1^2}\\
&=\frac{py_2^3+qs^3-y_1^2(py_2+qs)}{py_2^2+qs^2-y_1(py_2+qs)}=\frac{py_2(y_2^2-y_1^2)+qs(s^2-y_1^2)}{py_2(y_2-y_1)+qs(s-y_1)}.
\end{align*}
We observe that, 
\begin{align*}
\frac{py_2(y_2^2-y_1^2)+qs(s^2-y_1^2)}{py_2(y_2-y_1)+qs(s-y_1)} &\leq \frac{p s(s^2-y_1^2)+qs(s^2-y_1^2)}{ps(s-y_1)+qs(s-y_1)}\\
\leftrightarrow p^2 s y_2 (s-y_1)(y_2^2-y_1^2)&+qspy_2(y_2^2-y_1^2)(s-y_1)\\
\leq p^2 sy_2(s^2-y_1^2)(y_2-y_1) &+qspy_2(s^2-y_1^2)(y_2-y_1)\\
\leftrightarrow qspy_2(y_2-y_1)(s-y_1)(y_2-s)&\leq p^2 sy_2(y_2-y_1)(s-y_1)(s-y_2).
\end{align*}  
Thus, $f(x)\leq s+y_1$. As $p, q>0$, $-y_1>\frac{s}{n-1}$, so $f(x)<\frac{s(n-2)}{n-1}$, contradicting (\ref{atleast}). Thus, this case is ruled out.

\textbf{Case 2:} There are $m$ $y_1$'s and $q$ of $s$'s. Then we have,
\begin{align*}
my_1+qs &=0;\\
f(x)=\frac{x^3}{x^2} &=\frac{q s^3+my_1^3}{qs^2+my_1^2}\\
&=\frac{qs^3-y_1^2qs}{qs^2-y_1qs}=s+y_1=s(1-\frac{q}{m})\leq s(1-\frac{1}{n-1})=\frac{s(n-2)}{n-1}.
\end{align*}
That concludes the proof. 
\end{proof}
\subsection{Four-manifolds}
\label{fourdim}
In this section, we explain how computation in previous Sections simplifies greatly in dimension four. 

When $n=4$, $\WW=\WW^{+}+\WW^{-}$. First, we observe the following.
\begin{lemma} In dimension four,
\[\left\langle{\EE\circ g, \WW^2}\right\rangle= 0.\]
The statement also holds when replacing $\WW$ by $\WW^{\pm}$. 
\end{lemma}
\begin{proof}
It is well-known that, for example see \cite[Lemma 2.1]{caotran1}
\[\sum_{k,p,q}\WW_{ikpq}\WW_{j}^{~kpq}=|\WW|^2 g_{ij}.\]
Therefore, for an orthonormal frame diagonalizing both $g$ and $\EE$,
\[\left\langle{\EE\circ g, \WW^2}\right\rangle= \sum_{i,j,p,q}\EE_{ii}\WW_{ijpq}^2=|\WW|^2\sum_i \EE_{ii}=0.\]
\end{proof}

As a consequence, our BWL type formula (\ref{BWintegral}) becomes,
\begin{equation}
\label{bw4}
\int_M |\nabla \WW|^2-2\int_M (\delta\WW)^2=\int_M \left\langle{\WW,\WW^2+\WW^\sharp}\right\rangle -\frac{1}{2}\int_M \SS|\WW|^2.
\end{equation}
\begin{remark} Indeed, the same identity holds when replacing $\WW$ by $\WW^{\pm}$. 
\end{remark}
As $\WW^{\pm}$ is traceless and satisfies the first Bianchi identity, there is a normal form by M. Berger \cite{berger61, st69}. 
That is, there exists an orthonormal basis $\{e_{i}\}_{i=1}^4$ for $T_{p}M$, of which $\{e_{12}, e_{13}, e_{14}, e_{34}, e_{42}, e_{23}\}$ becomes a basis of $\Lambda_{2}$, such that
\begin{align}
\label{normalform}
\WW &=
 \left( \begin{array}{cc}
A & B \\
B & A \end{array} \right);\\
A &=\text{diag}(a_{1}, a_{2}, a_{3}), \nonumber\\
B &=\text{diag}(b_{1}, b_{2}, b_{3}),\nonumber\\
0 &=a_1+a_2+a_3=b_1+b_2+b_3.\nonumber
\end{align}
Then, the self-dual and anti-self-dual $\WW^{\pm}$ could be computed as,
\[\WW^{\pm}=
 \left( \begin{array}{cc}
\frac{A\pm B}{2} & \frac{B\pm A}{2} \\
\frac{B\pm A}{2} & \frac{A\pm B}{2} \end{array} \right).\]
Using this set-up, it is immediate to check that,
\begin{align*}
\left\langle{\WW^+, (\WW^+)^\sharp}\right\rangle &= \frac{1}{2}\sum_{i,j,k,l,p,q}\WW^+_{ijkl}\WW^+_{ipkq}\WW^+_{jplq}=6 \text{det}(\WW^+);\\
\left\langle{\WW^+, (\WW^+)^2}\right\rangle &=\frac{1}{8}\sum_{i,j,k,l,p,q}\WW^+_{ijkl}\WW^+_{ijpq}\WW^+_{klpq}=3\text{det}(\WW^+).
\end{align*}

Therefore, another way to write (\ref{bw4}) is,
\begin{equation}
\label{bw41}
\int_M |\nabla \WW^+|^2-2\int_M (\delta\WW^+)^2=\int_M 18\text{det}(\WW^+) -\frac{1}{2}\int_M \SS|\WW^+|^2.
\end{equation}
\begin{remark} This is the same as \cite[Equation 3.23]{CGY03} modulo out norm convention. 
\end{remark}
 
Furthermore, as the self-dual and the anti-self-dual part are operators on a vector space of dimension three, Lemma \ref{Wcubic} yields, for $\omega$ the largest eigenvalue of $\WW^+$,
 \[(\WW^+)^3 \leq \frac{\omega}{2}|\WW|^2.\]
 Consequently, we recover the following result which is originally due to \cite[Theorem 4.2]{mw93}.
\begin{theorem}
Let $(M^4,g)$ be a closed Riemannian manifold with harmonic self-dual Weyl tensor. $\omega$ denotes the largest eigenvalue of $\WW^+$. Suppose that, at each point, 
$$  6\omega \leq {\SS},$$ 
then $(M,g)$ has parallel Weyl tensor. 
\end{theorem}
It is interesting that another condition also leads to the desired rigidity.

\begin{lemma} \label{pinched} Let $\lambda_1\geq \lambda_2\geq \lambda_3$ be eigenvalues of $\WW^+$. Suppose that,
\begin{equation} 
\label{farapart} \lambda_1\geq \frac{\SS}{6} \text{ and } \lambda_3\leq-\frac{\lambda_1}{2}-\lambda_1 \sqrt{\frac{3(\lambda_1-\frac{\SS}{6})}{4(3\lambda_1+\frac{\SS}{6})}}, 
\end{equation}
then we have,
\begin{equation*}
\SS|\WW^+|^2\geq 36\text{det}(\WW^+).
\end{equation*}
\end{lemma}
\begin{proof}
Let $\mu_i=\lambda_i+\frac{\SS}{12}$. By direct computation, for $\SS_2=\sum_{i<j}\mu_i\mu_j, \SS_3=\prod_{i=1}^3 \mu_i$.
\[
\SS|\WW^+|^2- 36\text{det}(\WW^+)=\SS\SS_2-36\SS_3.\]
Without loss of generality, we can assume $\sum_{i=1}^3\mu_i=\frac{\SS}{4}=3$. Then we have,
\begin{align*}
\SS_2-3\SS_3 &= \mu_1\mu_3(1-3\mu_2)+(\mu_1+\mu_3)\mu_2,\\
&= \mu_1\mu_3-3\mu_1\mu_3(3-\mu_1-\mu_3)+(\mu_1+\mu_3)(3-\mu_1-\mu_3),\\
&=(3\mu_1-1)\mu_3^2+(3\mu_1-1)(\mu_1-3)\mu_3+(3-\mu_1)\mu_1.
\end{align*}
Next, we compute the determinant of that quadratic on $\mu_3$,
\begin{align*}
\Box &=(3\mu_1-1)^2(\mu_1-3)^2-4(3\mu_1-1)(3-\mu_1)\mu_1,\\
&=(3\mu_1-1)(\mu_1-3)(3\mu_1^2-10\mu_1+3+4\mu_1),\\
&=3(3\mu_1-1)(\mu_1-3)(\mu_1-1)^2
\end{align*}
Since $\mu_1> 3$, the quadratic above is positive if and only if $\mu_3 \notin [x_1,x_2]$ where 
\[x_{1,2}=-\frac{\mu_1-3}{2}\pm(\mu_1-1)\sqrt{\frac{3(\mu_1-3)}{4(3\mu_1-1)}} \] 
As $\mu_3<0$ the only possibility is that $\mu_3\leq x_1$. The statement now follows from normalization and the fact that $\mu_i=\frac{\SS}{12}+\lambda_i$. 
\end{proof}
\begin{remark}
As a consequence, we also obtain a rigidity result assuming condition (\ref{farapart}) and harmonic Weyl tensor. 
\end{remark}

Next, we consider the proof of the integral rigidity statement, Theorem \ref{integralrigid}, in dimension four. First, without loss of generality, we can choose to work with $\WW^{+}$. 
Due to the discussion above, we have the sharp estimate:
\begin{equation}
2\left\langle{\WW,\WW^2+\WW^\sharp}\right\rangle =18\text{det}(\WW^+)\leq \sqrt{6}|\WW^{+}|^3.
\end{equation}
The equality happens only if either $\WW^+=0$ or it has exactly two distinct eigenvalues. 

Thus, $c_1(4)=\sqrt{6}$. Then proceeding as earlier, by H\"{o}lder and Sobolev inequalities, we obtain $\alpha=\frac{1}{3}$.  As a result, rigidity is obtained for a manifold with harmonic self-dual and the following inequality,
\[||\WW||_{L^{2}}\leq \frac{1}{2\sqrt{6}}\lambda(g).
\]
 \begin{remark} This result, however, is not new, see \cite{g00, fu16four}. And the inequality is sharp. \end{remark}
  
\subsection{Pure curvature}
\label{pure}
The pure curvature condition says that there is an orthonormal frame $\{e_i\}_{i=1}^n$at each point such that the curvature is diagonalized with respect to the basis $\{e_i\wedge e_j\}_{i\neq j}$. Then, the only nontrivial component of the Weyl tensor is $\WW_{ijij}\doteqdot w_{ij}$. It is observed that $w$ is a symmetric matrix with the following properties:
\begin{itemize}
\item Each diagonal term is zero;
\item The sum of each row is zero. 
\end{itemize}

In this case, we can compute $\left\langle{\WW,\WW^{2}}\right\rangle$ and $\left\langle{\WW,\WW^{\sharp}}\right\rangle$ as follows. 
\begin{align*}
\left\langle{\WW, \WW^\sharp}\right\rangle &= \sum_{i,j,k,l,p,q}\WW_{ijkl}\WW_{ipkq}\WW_{jplq}=6 \sum_{i,j,k}w_{ij}w_{ik}w_{kj};\\
\left\langle{\WW, \WW^2}\right\rangle &=\frac{1}{4}\sum_{i,j,k,l,p,q}\WW_{ijkl}\WW_{ijpq}\WW_{klpq}=2\sum_{ij}w_{ij}^3.
\end{align*} 
 
Recalling the following identity,
\[a^3+b^3+c^3-3abc=\frac{1}{2}(a+b+c)(3a^2+3b^2+3c^2-(a+b+c)^2).\]
By applying for each term $w_{ij}w_{ik}w_{kj}$,
\begin{align*}
\left\langle{\WW, \WW^\sharp}\right\rangle &=6 \sum_{i,j,k}w_{ij}w_{ik}w_{kj},\\
 &= \sum_{i\neq j\neq k}2(w_{ij}^3+w_{ik}^3+w_{jk}^3)+(w_{ij}+w_{jk}+w_{ik})^3\\
 &~~~~~-3(w_{ij}+w_{jk}+w_{ik})(w_{ij}^2+w_{ik}^2+w_{jk}^2),\\
&=2(n-2)\sum_{ij}w_{ij}^3+\sum_{i\neq j\neq k}(w_{ij}+w_{jk}+w_{ik})^3-3(n-4)\sum_{ij}w_{ij}^3,\\
&=\frac{8-n}{2}\left\langle{\WW, \WW^2}\right\rangle+\sum_{i\neq j\neq k}(w_{ij}+w_{jk}+w_{ik})^3
\end{align*} 
By (\ref{weylsect}), the last term is just a summation of the Weyl generalized sectional curvature on 3-planes. In dimension four, as $4-3=1$, that term vanishes. In dimension five, 
\[\sum_{i\neq j\neq k}(w_{ij}+w_{jk}+w_{ik})^3=\sum_{mn}w_{mn}^3=\frac{1}{2}\left\langle{\WW, \WW^2}\right\rangle. 
\]
Thus, when $n\leq 5$, $\left\langle{\WW, \WW^\sharp}\right\rangle =2\left\langle{\WW, \WW^2}\right\rangle$. There is little evidence to suggest the equation still holds true in higher dimension.






\def\cprime{$'$}
\bibliographystyle{plain}
\bibliography{bio}

\end{document}